\documentclass[11pt]{amsart}

\usepackage[a4paper,hmargin=3.5cm,vmargin=3.5cm]{geometry}
\usepackage{amsfonts,amssymb,amscd,amstext}
\usepackage{graphicx}
\usepackage[dvips]{epsfig}
\usepackage{quoting}
\usepackage{hyperref}

\usepackage{fancyhdr}
\input xy
\xyoption{all}

\usepackage{enumerate}
\usepackage{titlesec}
\usepackage{mathrsfs}

\titleformat{\section}
{\filcenter\bfseries\large} {\thesection{.}}{0.2cm}{}
\titleformat{\subsection}[runin]
{\bfseries} {\thesubsection{.}}{0.15cm}{}[.]
\titleformat{\subsubsection}[runin]
{\em}{\thesubsubsection{.}}{0.15cm}{}[.]

\usepackage[up,bf]{caption}

\newtheorem{theoremA}{Theorem}

\newtheorem{theorem}{Theorem}[section]

\newtheorem{claim}[theorem]{Claim}
\newtheorem{corollary}[theorem]{Corollary}

\theoremstyle{definition}
\newtheorem{definition}[theorem]{Definition}
\newtheorem{remark}[theorem]{Remark}

\numberwithin{equation}{section}
\numberwithin{figure}{section}

\newcommand\Gcal{\mathcal{G}}

\newcommand\Fcal{\mathcal{F}}

\newcommand\C{\mathbb{C}}

\renewcommand\b{\mathbb{B}}
\renewcommand\c{\mathbb{C}}

\newcommand\n{\mathbb{N}}
\renewcommand\r{\mathbb{R}}

\newcommand\z{\mathbb{Z}}

\newcommand\wt{\widetilde}

\newcommand\dist{\mathrm{dist}}

\newcommand\length{\mathrm{length}}

\def\dist{\mathrm{dist}}

\def\length{\mathrm{length}}

\def\diam{\mathrm{diam}}

\usepackage{color}

\begin{document}


\fancyhead[LO]{Wild holomorphic foliations of the ball}
\fancyhead[RE]{A.\ Alarc\'on}
\fancyhead[RO,LE]{\thepage}

\thispagestyle{empty}


\begin{center}

{\bf\LARGE Wild holomorphic foliations of the ball}

\medskip

%
%
{\large\bf Antonio Alarc\'on}
\end{center}


%
%
\medskip

\begin{quoting}[leftmargin={5mm}]
{\small
\noindent {\bf Abstract}\hspace*{0.1cm}
We prove that the open unit ball $\b_n$ of $\c^n$ $(n\ge 2)$ admits a nonsingular holomorphic foliation $\Fcal$ by closed complex hypersurfaces such that both the union of the complete leaves of $\Fcal$ and the union of the incomplete leaves of $\Fcal$ are dense subsets of $\b_n$. In particular, every leaf of $\Fcal$ is both a limit of complete leaves of $\Fcal$ and a limit of incomplete leaves of $\Fcal$. This gives the first example of a holomorphic foliation of $\b_n$ by connected closed complex hypersurfaces having a complete leaf that is a limit of incomplete ones. 
We obtain an analogous result for foliations by complex submanifolds of arbitrary pure codimension $q$ with $1\le q<n$. 


\noindent{\bf Keywords}\hspace*{0.1cm} 
complete Riemannian manifold,
complex submanifold,
holomorphic foliation,
noncritical holomorphic function, 
holomorphic submersion,
limit leaf.


\noindent{\bf Mathematics Subject Classification (2020)}\hspace*{0.1cm} 
32H02, 
53C12 
}
\end{quoting}



\section*{Introduction}
\label{sec:intro}


\noindent 
Motivated by the Yang problem on the existence of complete bounded complex submanifolds of the complex Euclidean spaces \cite{Yang1977,Yang1977JDG}, Globevnik constructed in \cite{Globevnik2015AM} holomorphic functions $f$ on the open unit ball $\b_n$ of $\c^n$ $(n\ge 2)$ all of whose nonempty level sets $f^{-1}( z)$ $(z\in\c)$ 
are complete as metric spaces; thus, the family of their components forms a holomorphic foliation of $\b_n$ by complete connected closed (possibly singular) complex hypersurfaces. Most of the leaves of these foliations are smooth by Sard's theorem; nevertheless, there is such a function $f$ which is noncritical, hence $\b_n$ admits a nonsingular holomorphic foliation by complete properly embedded complex hypersurfaces \cite{Alarcon2018Foliations}. 

The ball $\b_n$ also carries holomorphic foliations by closed complex hypersurfaces having both complete and incomplete leaves.
In fact, every closed complex hypersurface in $\b_n$, not necessarily complete or connected, can be embedded as a union of leaves into a holomorphic foliation $\Fcal$ of $\b_n$ by connected closed complex hypersurfaces all of whose leaves, except perhaps those in the given hypersurface, are complete \cite[Theorem 1.4]{Alarcon2018Foliations}. 
Note that the foliations constructed in \cite{Alarcon2018Foliations} 
have {\em few} incomplete leaves; for instance, the family of incomplete leaves of any such $\Fcal$
admits no limit leaves (that is, no leaf of $\Fcal$ is a limit leaf of the incomplete leaves of $\Fcal$). 

In this paper we give
new examples of nonsingular holomorphic foliations of the ball by closed complex hypersurfaces having both complete and incomplete leaves. In particular, we construct such foliations with {\em many} incomplete leaves. 
Here is a sampling of the results we obtain.
\begin{theoremA}\label{th:A}
The open unit ball $\b_n$ of $\c^n$ $(n\ge 2)$ admits a nonsingular holomorphic foliation $\Fcal$ by
connected properly embedded complex hypersurfaces such that both the union of the complete leaves of $\Fcal$ and the union of the incomplete leaves of $\Fcal$ are dense subsets of $\b_n$.
\end{theoremA}

Note that every leaf of the foliation $\Fcal$ in the theorem is a limit leaf of both the family of complete leaves of $\Fcal$ and the family of incomplete leaves of $\Fcal$. This gives the first example of a foliation of $\b_n$ by connected closed complex hypersurfaces having a complete leaf that is a limit of incomplete ones.

We go further and provide foliations of $\b_n$, with the same properties as above, by closed complex submanifolds of arbitrary pure codimension $q$ with $1\le q<n$. The foliations we construct are formed by the fibres of holomorphic submersions $\b_n\to\c^q$; we give precise statements of our main results 
in Section \ref{sec:MR}. Moreover, it is clear that our results can be easily adapted to any bounded pseudoconvex Runge domain of $\c^n$ in place of the ball $\b_n$, for one just has to combine our construction method with the use of the holomorphically convex labyrinths of compact sets constructed in such domains by Charpentier and Kosi\'nski in \cite{CharpentierKosinski2020}. The fact that every pseudoconvex domain of $\c^n$ admits a (possibly singular) holomorphic foliation by complete closed complex hypersurfaces was proved by Globevnik in \cite{Globevnik2016MA}. Every Stein manifold of dimension $>1$ equipped with an arbitrary Riemannian metric too, as was subsequently shown in \cite{Alarcon2018Foliations}. The first construction of nonsingular holomorphic foliations of $\b_n$ by smooth complete closed complex submanifolds of arbitrary pure codimension can be found also in \cite{Alarcon2018Foliations}. We refer to Alarc\'on and Forstneri\v c \cite[Sec.\ 1]{AlarconForstneric2020MZ} for an up-to-date survey on this topic.


\section{Main results and outline of the paper
}\label{sec:MR}

\noindent Given a holomorphic submersion $f\colon X\to\c^q$ on a complex manifold $X$ with $\dim X>q\ge 1$,  the family of components of the nonempty fibres $f^{-1}(z)$ $(z\in \c^q)$ of $f$ form a nonsingular holomorphic foliation of $X$ by smooth closed (properly embedded) connected complex submanifolds of codimension $q$. A foliation of $X$ obtained in this way is called a {\em (codimension-$q$) holomorphic submersion foliation}.

Here is the first main result of this paper.
\begin{theorem}\label{th:intro-main}
Let  $n$ and $q$ be integers with $1\le q<n$. For any pair of disjoint countable sets $A$ and $C$ in $\c^q$ there is a surjective holomorphic submersion $f\colon\b_n\to\c^q$ satisfying the following properties.
\begin{itemize}
\item[\rm (i)] If $a\in A$, then no component of the fibre $f^{-1}(a)\subset \b_n$ is complete.

\smallskip
\item[\rm (ii)] The fibre $f^{-1}(c)\subset \b_n$ is complete for all $c\in C$.
\end{itemize}
\end{theorem}
Since the submersion $f$ in the theorem is surjective, $f^{-1}(z)$ is empty for no $z\in \c^q$, hence for no $z\in A\cup C$. 
Thus, choosing $A$ and $C$ to be nonempty,
Theorem \ref{th:intro-main} provides codimension-$q$ holomorphic submersion foliations of the ball $\b_n$ having both complete and incomplete leaves. In particular, it gives the first example of such a foliation whose family of incomplete leaves admits limit leaves (see Corollary \ref{co:dense} for a more precise result; see also Corollary \ref{co:infinitely} for another result in this direction). Recall that if $\Fcal$ is a foliation of a manifold $X$ and $\Gcal\subset \Fcal$ is a family of leaves of $\Fcal$, then a leaf $V\in \Fcal$ (possibly $V\notin\Gcal$) is said to be a {\em limit leaf} of $\Gcal$ if $V$ is contained in the topological  closure of $\Gcal\setminus \{V\}$ as subset of $X$; that is, if $V\subset \overline{\bigcup_{W\in\Gcal\setminus\{V\}}W}$. 

Choosing the sets $A$ and $C$ in Theorem \ref{th:intro-main} to be both dense subsets of $\c^q$ we obtain the following result that strengthens Theorem \ref{th:A}.
\begin{corollary}\label{co:dense}
For any pair of integers $n$ and $q$ with $1\le q<n$ there is
 a codimension-$q$ holomorphic submersion foliation $\Fcal$ of $\b_n$ such that both the union of the complete leaves of $\Fcal$ and the union of the incomplete leaves of $\Fcal$ are dense subsets of $\b_n$. In particular, every leaf of $\Fcal$
is both a limit leaf of the family of complete leaves of $\Fcal$ and a limit leaf of the family of incomplete leaves of $\Fcal$.
\end{corollary}

For $z\in \c^q\setminus(A\cup C)$ we do not know whether the fibre $f^{-1}(z)\subset\b_n$ of the submersion $f$ provided by Theorem \ref{th:intro-main} is complete or incomplete. Regarding this concern, we also prove in this paper the following result in which the completeness character of all the fibres of $f$ is controlled at the cost of assuming extra conditions on the sets $A$ and $C$.
\begin{theorem}\label{th:intro}
Let  $n$ and $q$ be integers with $1\le q<n$ and let $\wt A$ and $\wt C$ be a pair of disjoint closed discrete subsets of $\c^q$. Let $A$ be the union of $\wt A$ with a closed discrete subset of $\c^q\setminus (\wt A\cup\wt C)$ and call $C=\c^q\setminus A$. Then the conclusion of Theorem \ref{th:intro-main} holds; that is, there is a surjective holomorphic submersion $f\colon\b_n\to\c^q=A\cup C$ satisfying the following properties.
\begin{itemize}
\item[\rm (i)]  If $a\in A\supset\wt A$, then no component of the fibre $f^{-1}(a)\subset\b_n$ is complete.

\smallskip
\item[\rm (ii)] The fibre $f^{-1}(c)\subset\b_n$ is complete for all $c\in C\supset \wt C$.
\end{itemize}
Thus, the family of connected components of the fibres $f^{-1}(z)$ $(z\in \c^q)$ of $f$ is a codimension-$q$ holomorphic submersion foliation of $\b_n$ all of whose leaves, except precisely those in fibres over points in $A$, are complete. 
\end{theorem}
In the special cases when the set $A\subset\c^q$ is either empty or unitary, Theorem \ref{th:intro} (which in these particular instances implies Theorem \ref{th:intro-main}) follows from the results in \cite{Alarcon2018Foliations} except for the property that the holomorphic submersion $f\colon\b_n\to\c^q$ is surjective, which, nevertheless, is not difficult to ensure. The role of this condition in Theorems \ref{th:intro-main} and \ref{th:intro} is to guarantee that the sets $A$ and $C$ lie in the image of the holomorphic submersion $f$, and hence the fibre $f^{-1}(z)$ is empty for no $z\in A\cup C$.

The first assertion in the following corollary to Theorem \ref{th:intro} is easily obtained from the results in \cite{Alarcon2018Foliations},
although it is not stated there. 
The second assertion is new.
\begin{corollary}\label{co:infinitely}
For any pair of integers $n$ and $q$ with $1\le q<n$ there is a codimension-$q$ holomorphic submersion foliation of $\b_n$ all of whose leaves, except precisely countably-infinitely many among them, are complete. 

Furthermore, there is such a foliation whose family of incomplete leaves admits both complete limit leaves and incomplete limit leaves.
\end{corollary}
%
%
\begin{proof}
It suffices to apply Theorem \ref{th:intro} with subsets $\wt A=\{a\}$, $\wt C=\{c\}$, and $A=\wt A\cup\{a_0,a_1,a_2,\ldots\}$ of $\c^q$, where $a\neq c$ and $a_0,a_1,a_2,\ldots$ is a sequence of pairwise different points in $\c^q\setminus\{a,c\}$ whose limit set is $\wt A\cup\wt C=\{a,c\}$. We may for instance choose $a_0,a_1,a_2,\ldots$ with $\lim_{j\to\infty}a_{2j+1}=a$ and $\lim_{j\to\infty}a_{2j}=c$. 
\end{proof}

%
%
We prove Theorem \ref{th:intro-main} in Section \ref{sec:proof-main}. 
As in \cite{Alarcon2018Foliations}, we shall obtain a holomorphic submersion $f\colon \b_n\to\c^q$ satisfying the conclusion of the theorem as the limit of a sequence of holomorphic submersions $f_j\colon \b_n\to\c^q$ $(j\in\n)$. The main new ingredient in the construction in the present paper is that we ensure in the recursive process that every component of the fibre $f^{-1}(a)$ of $f$ over each point $a$ in the set $A$ extends to a smooth closed complex submanifold  of codimension $q$ in $\c^n$, thereby guaranteeing that no component of $f^{-1}(a)\subset\b_n$ is complete by the boundedness of $\b_n$. In order to ensure that property of the limit submersion $f\colon\b_n\to\c^q$ we choose each $f_j$ extending to a holomorphic submersion on $\c^n$. Nevertheless, since the fibres of $f$ over all points in $C$ are required to be complete, the sequence $f_j\colon \c^n\to\c^q$ shall only converge uniformly in compact subsets of $\b_n$. 
The main technical tool in our construction is an Oka-Weil-Cartan type theorem for holomorphic submersions from any Stein manifold $X$ to $\c^q$ $(1\le q<\dim X)$ due to Forstneri\v c (see \cite{Forstneric2003AM} or \cite[\textsection 9.12-9.16]{Forstneric2017}).
The proof of Theorem \ref{th:intro}, which is similar to that of Theorem \ref{th:intro-main}, is explained in Section \ref{sec:proof}.

%
%
\begin{remark}
Our method of proof also ensures approximation on compact subsets of the ball. To be precise, an inspection of the proofs in this paper shows that if we are given a holomorphic submersion $f_0\colon \C^n\to\C^q$ $(1\le q<n)$ and a polynomially convex compact set $K\subset\b_n$, then there is a holomorphic submersion $f\colon\b_n\to\c^q$ satisfying the conclusion of Theorem \ref{th:intro-main} (respectively, Theorem \ref{th:intro}) which is as close as desired to $f_0$ uniformly on $K$. 
\end{remark}
\begin{remark}
It is clear that Theorems \ref{th:intro-main} and  \ref{th:intro} can be adapted to  any bounded pseudoconvex Runge domain of $\c^n$ in place of the ball $\b_n$. Indeed, one just has to adapt the proofs by replacing our tangent labyrinths (see Definition \ref{def:labyrinth}) by those labyrinths of compact sets constructed by Charpentier and Kosi\'nski in \cite{CharpentierKosinski2020}. The boundedness of the domain is used to ensure the incompleteness of some of the leaves of the foliation; see \eqref{eq:component}. However, this assumption is not completely necessary: for instance, it suffices to ask the domain to contain no smooth closed complex submanifold of $\c^n$.
\end{remark}
\begin{remark}
The seminal construction of Globevnik \cite{Globevnik2015AM} gives a holomorphic function on $\b_n$ $(n\ge 2)$ whose real part is unbounded on every divergent path $\gamma\colon [0,1)\to \b_n$ with finite length; it follows that every level set of such a function is a complete closed (possibly singular) complex hypersurface in $\b_n$. Charpentier and Kosi\'nski \cite{CharpentierKosinski2021} extended Globevnik's result by constructing holomorphic functions $f$ on $\b_n$ with the property that $f(\gamma([0,1)))$ is everywhere dense in $\c$ for any path $\gamma$ as above, and showed that such functions are in fact abundant: they form a residual, densely lineable, and spaceable subset of the space of all holomorphic functions on $\b_n$ endowed with the locally uniform convergence topology. (These results hold with any pseudoconvex domain in place of the ball \cite{Globevnik2016MA,CharpentierKosinski2021}. The author thanks the referee for drawing his attention to the paper \cite{CharpentierKosinski2021}.) We expect that the submersions given in Theorems \ref{th:intro-main} and  \ref{th:intro} are abundant too. 
\end{remark}


\section{Notation and preliminaries}\label{sec:prelim}
  
\noindent For a set $A$ in a topological space $X$ we denote by $\overline A$, $\mathring A$, and $b A=\overline A\setminus \mathring A$ the topological closure, interior, and frontier of $A$ in $X$, respectively. For $B\subset X$ we write $A\Subset B$ when $\overline A\subset\mathring B$. 
 A point $x\in X$ is called a {\em limit point} of a subset $A\subset X$ if every neighborhood of $x$ in $X$ intersects $A\setminus\{x\}$. The set $A\subset X$ is called {\em discrete} if no point of $A$ is a limit point of $A$ itself; that is, every point $a\in A$ admits a neighborhood in $X$ which is disjoint from $A\setminus\{a\}$. 
A subset $A\subset X$ is closed and discrete if and only if no point of $X$ is a limit point of $A$. 

Let $\Fcal$ be a foliation of a manifold $X$
and let $\Gcal\subset \Fcal$ be a family of leaves of $\Fcal$. A leaf $V\in \Fcal$ (possibly, $V\notin\Gcal$) is said to be a {\em limit leaf} of $\Gcal$ if $V$ lies in the topological closure of $\Gcal\setminus\{V\}$ as subset of $X$; that is, if $V\subset \overline{\bigcup_{W\in\Gcal\setminus\{V\}}W}$. This is equivalent to that every point of $V$ is a limit point of $\Gcal\setminus\{V\}$ as subset of $X$. We emphasize that, in this definition, the family $\Gcal$ is not assumed to be closed as subset of $X$; that is, it is not required that $\bigcup_{W\in\Gcal}W\subset X$ be a closed set. 

We write $\n=\{1,2,3,\ldots\}$ and $\z_+=\n\cup\{0\}$. We shall denote by $|\cdot|$, $\dist(\cdot,\cdot)$, $\length(\cdot)$, and $\diam(\cdot)$ the Euclidean norm, distance, length, and diameter in $\r^n$ for any $n\in\n$. We denote by $\b_n=\{z\in\c^n\colon |z|<1\}$ the open unit ball in $\c^n$ for any integer $n\ge 2$. Given a point $z\in\c^n$, a set $A\subset\c^n$, and a number $\zeta\in\c$, we write $z+\zeta A=\{z+\zeta a\colon a\in A\}$.

In the proof of Theorems \ref{th:intro-main} and \ref{th:intro} we shall use as a key ingredient the labyrinths of compact sets in open spherical shells of $\c^n$ that were first introduced in \cite{AlarconGlobevnikLopez2016Crelle}; see also \cite{AlarconGlobevnik2017C2,AlarconForstneric2017PAMS}. 
%
%
\begin{definition}[\text{\cite[Def.\ 2.2]{Alarcon2018Foliations}}]\label{def:labyrinth}
A compact set $L$ contained in an open spherical shell $R\b_n\setminus r\overline\b_n=\{z\in\c^n\colon r<|z|<R\}\subset\c^n$ $(n\ge 2$,  $0<r<R$) is called a {\em tangent labyrinth} if $L$ has finitely many connected components, $T_1,\ldots,T_l$ $(l\in\n)$, and $L$ is the support of a finite tidy collection of tangent balls in the sense of \cite[Def.\ 1.3 and 1.4]{AlarconGlobevnikLopez2016Crelle}; equivalenty, $L$ satisfies the following conditions.
\begin{itemize}
\item Each component $T_j$ of $L$ is a closed round ball in a real affine hyperplane in $\c^n=\r^{2n}$ which is orthogonal to the position vector of the center $x_j$ of the ball $T_j$. 

\smallskip
\item If for some $i,j\in\{1,\ldots,l\}$ the centers $x_i$ of $T_i$ and $x_j$ of $T_j$ satisfy $|x_i|=|x_j|$, then the radii of $T_i$ and $T_j$ are equal. If $|x_i|<|x_j|$, then $T_i\subset |x_j|\b_n$.
\end{itemize}
\end{definition}

A smooth closed complex submanifold $V$ in $\b_n$ $(n\ge 2)$ is said to be {\em complete} if the Riemannian metric $g$ induced on $V$ by the Euclidean one in $\c^n$ is complete in the classical sense. That is, the Riemannian manifold $(V,g)$ is a complete metric space, meaning that Cauchy sequences are convergent; equivalently, $(V,g)$ is geodesically complete in the sense that every geodesic $\alpha(t)$ in $(V,g)$ is defined for all values of time $t\in\r$. Completeness of $V$ is equivalent to that every path $\gamma \colon [0,1)\to\b_n$ with $\gamma([0,1))\subset V$ and $\lim_{t\to1}|\gamma(t)|=1$ have infinite Euclidean length. Such a path is called a {\em divergent} path or a {\em proper} path in $V$.
If $V$ is not complete, then it is said to be {\em incomplete}. We refer, for instance, to the monograph by do Carmo \cite{doCarmo1992} for an introduction to Riemannian geometry.

Let $X$ be a Stein manifold. 

A complex-valued function on a subset $A\subset X$ is said to be {\em holomorphic} if it is holomorphic in an unspecified open set in $X$ containing $A$. 
A holomorphic map $f=(f_1,\ldots,f_q)\colon X\to\c^q$ $(1\le q\le \dim X)$ is said to be {\em submersive} at a point $x\in X$ if the differential $df_x\colon T_xX\to T_{f(x)}\c^q\cong \c^q$ is surjective; that is, if the differentials of the component functions of $f$ at the point $x$ are linearly independent: 
$
	(df_1\wedge\cdots\wedge df_q)|_x\neq 0.
$
 The map $f$ is called a {\em submersion} if it is submersive at all points in $X$. 
 
\begin{definition}[\text{\cite[p.\ 148]{Forstneric2003AM}}]\label{def:qcoframe}
Let $q\in\n$. A {\em $q$-coframe} on a Stein manifold $X$ is  a $q$-tuple of continuous differential $(1,0)$-forms $\theta=(\theta_1,\ldots,\theta_q)$ on $X$ which are pointwise linearly independent at every point of $X$:
\[
	(\theta_1\wedge\cdots\wedge \theta_q)|_x\neq 0\quad \text{for all $x\in X$}.
\]
\end{definition}
If $X$ admits a $q$-coframe, then $\dim X\ge q$. If $f=(f_1,\ldots,f_q)\colon X\to\c^q$ is a holomorphic submersion, then the differential $df=(df_1,\ldots,df_q)$ is a $q$-coframe on $X$. Conversely, if $X$ admits a $q$-coframe for some integer $q$ with  $1\le q<\dim X$, then $X$ carries a holomorphic submersion to $\c^q$ (see  \cite[Theorem 2.5]{Forstneric2003AM}). 

If we are given a holomorphic submersion $f\colon X\to\c^q$, then the fibres $f^{-1}(z)$ $(z\in\c^q)$ of $f$ form a nonsingular holomorphic foliation $\Fcal$ of $X$ by smooth closed complex submanifolds of pure codimension $q$. A foliation $\Fcal$ obtained in this way is called a {\em codimension-$q$ holomorphic submersion foliation} of $X$ or just a {\rm holomorphic submersion foliation} of $X$ when the codimension is clear from the context. It turns out that if $z\in f(X)$ and $\{z_1,z_2,z_3,\ldots\}\subset f(X)\setminus\{z\}$ is a sequence with $\lim_{j\to\infty} z_j=z$, then every component of $f^{-1}(z)$ is a limit leaf of the family of leaves $\Gcal=\{W\in\Fcal\colon W\subset f^{-1}(z_j)\text{ for some } j\in\n\}$.


\section{Proof of Theorem \ref{th:intro-main}}\label{sec:proof-main}

\noindent Let  $n$ and $q$ be integers with $1\le q<n$ and let $A$ and $C$ be a pair of disjoint countable sets in $\c^q$. We assume without loss of generality that both $A$ and $C$ are infinite, hence bijective to $\z_+$. We establish an order in each of these sets and write 
\[
	A=\{a_0,a_1,a_2,\ldots\}
	\quad\text{and}\quad 
	C=\{c_0,c_1,c_2,\ldots\}.
\]
Choose a sequence of positive numbers $0<\lambda_0<\lambda_1<\lambda_2<\cdots$ with
\begin{equation}\label{eq:lambda>}
	\lambda_j>\max\{|a_0|,\ldots,|a_j|,|c_0|,\ldots,|c_j|\}\quad \text{for all $j\in\z_+$}
\end{equation}
and
\begin{equation}\label{eq:limnu}
	\lim_{j\to\infty}\lambda_j=+\infty.
\end{equation}
Also let $0<r_0<R_0<\rho_0<r_1<R_1<\rho_1<r_2<R_2<\rho_2<\cdots$ be a sequence of positive numbers with
\begin{equation}\label{eq:limrj}
	\lim_{j\to\infty}r_j=1.
\end{equation}
It follows that $\lim_{j\to\infty}R_j=\lim_{j\to\infty}\rho_j=1$ as well. 
Fix an open Euclidean ball $\Omega_0$ in $\c^n$ with $\overline \Omega_0\subset \rho_0\b_n\setminus R_0\overline\b_n$ and choose a holomorphic submersion $f_0\colon \c^n\to\c^q$ with
\begin{equation}\label{eq:f0}
	\lambda_0\overline\b_q\subset f_0(\Omega_0).
\end{equation}
Such an $f_0$ clearly exists; in fact, since holomorphic submersions $\c^n\to\c^q$ are open maps, every such a submersion satisfies the required condition up to post-composition with a suitable affine transformation of $\c^q$. 
Also fix a positive number $\epsilon_0>0$  and set $L_0=\varnothing$ and $O_0=\varnothing$. 

The recursive construction of a holomorphic submersion $f\colon\b_n\to\c^q$ satisfying the conclusion of Theorem \ref{th:intro-main} is enclosed in the following assertion.
\begin{claim}\label{cl:induction}
There is a sequence $S_j=\{f_j,\epsilon_j,L_j,O_j,\Omega_j\}$ $(j\in\n)$, where
\begin{itemize}
\item $f_j\colon \c^n\to\c^q$ is a holomorphic submersion,

\smallskip
\item $\epsilon_j$ is a positive number,

\smallskip
\item $L_j$ is a tangent labyrinth in the open spherical shell $R_j\b_n\setminus r_j\overline\b_n$ (see Definition \ref{def:labyrinth}),

\smallskip
\item $O_j$ is an open neighborhood of $L_j$ in $\c^n$ such that $\overline O_j\subset R_j\b_n\setminus r_j\overline\b_n$, and

\smallskip
\item $\Omega_j$ is an open Euclidean ball in $\c^n$ with $\overline \Omega_j\subset \rho_j\b_n\setminus R_j\overline\b_n$,
\end{itemize}
such that the following conditions are satisfied for all $j\in\n$.
\begin{itemize}
\item[\rm (1$_j$)] $|f_j(z)-f_{j-1}(z)|<\epsilon_j$ for all $z\in r_j\overline\b_n$.

\smallskip
\item[\rm (2$_j$)] $\lambda_i\overline\b_q\subset f_j(\Omega_i)$ for all $i\in\{0,\ldots,j\}$.

\smallskip
\item[\rm (3$_j$)] $f_j^{-1}(a_i)\cap r_j\overline\b_n=f_{j-1}^{-1}(a_i)\cap r_j\overline\b_n$ for all $i\in\{0,\ldots,j\}$. 

\smallskip
\item[\rm (4$_j$)] $0<\epsilon_j<\epsilon_{j-1}/2$.

\smallskip
\item[\rm (5$_j$)] If $\phi\colon \b_n\to\c^q$ is a holomorphic map such that $|\phi(z)-f_{j-1}(z)|<2\epsilon_j$ for all $z\in r_j\overline\b_n$, then $\phi$ is submersive at every point in $\rho_{j-1}\overline\b_n$.

\smallskip
\item[\rm (6$_j$)] If $\gamma\colon[0,1]\to \b_n$ is a path satisfying $|\gamma(0)|\le r_j$, $|\gamma(1)|\ge R_j$, and $\gamma([0,1])\cap L_j=\varnothing$, then $\length(\gamma)>1$.

\smallskip
\item[\rm (7$_j$)] $f_j(\overline O_i)\cap \{c_0,\ldots,c_i\}=\varnothing$ for all $i\in\{0,\ldots,j\}$. 
\end{itemize}
\end{claim}

We defer the proof of Claim \ref{cl:induction} to later on. Granted the claim, the proof of  Theorem \ref{th:intro-main} is completed as follows.

Let $S_j=\{f_j,\epsilon_j,L_j,O_j,\Omega_j\}$ $(j\in\n)$ be a sequence of tuples provided by Claim \ref{cl:induction}.
It follows from  \eqref{eq:limrj} and properties {\rm (1$_j$)} and {\rm (4$_j$)} that the sequence of holomorphic submersions $f_j\colon\c^n\to\c^q$ $(j\in\n)$ converges uniformly on compact sets in $\b_n$ to a holomorphic map
\begin{equation}\label{eq:f=limfj}
	f=\lim_{j\to\infty}f_j\colon\b_n\to\c^q
\end{equation}
satisfying
\begin{multline}\label{eq:f-fj-1}
	|f(z)-f_{j-1}(z)| \le \sum_{i\ge j}|f_i(z)-f_{i-1}(z)| <\sum_{i\ge j} \epsilon_i < \epsilon_j(1+\sum_{i=1}^{+\infty}  \frac1{2^i}) =
	\\
	=2\epsilon_j<\epsilon_{j-1}\quad \text{for all $z\in r_j\overline\b_n$, $j\in\n$}.
\end{multline}
It turns out that $f$ is a holomorphic submersion by \eqref{eq:f-fj-1} and properties {\rm (5$_j$)}.
(Despite each $f_j$ is defined everywhere on $\c^n$, the convergence of this sequence is guaranteed only on $\b_n$ in view of \eqref{eq:limrj} and properties {\rm (1$_j$)} and {\rm (4$_j$)}. Therefore, the limit map $f$ in \eqref{eq:f=limfj} is ensured to be defined only on the open unit ball $\b_n$; that is not a misprint.) 

We claim that the holomorphic submersion $f$ in \eqref{eq:f=limfj} satisfies the conclusion of Theorem \ref{th:intro-main}.
Indeed, note first that 
\begin{equation}\label{eq:surjective}
	\lambda_i\overline\b_q\subset f(\overline \Omega_i)\subset f(\b_n)\quad \text{for all $i\in\z_+$}.
\end{equation}	
Indeed, choose $i\in\z_+$ and a point $z_0\in \lambda_i\overline\b_q$. By \eqref{eq:f0} and properties {\rm (2$_j$)}, $j\in\n$, there is for each $j\ge i$ a point $\zeta_j\in \Omega_i$ with $f_j(\zeta_j)=z_0$. By compactness of $\overline\Omega_i$ we may assume, up to passing to a subsequence if necessary, that $\lim_{j\to\infty}\zeta_j=\zeta\in\overline\Omega_i$. It turns out that
\begin{eqnarray*}
	|f(\zeta)-z_0| & = & |f(\zeta)-f_j(\zeta_j)|
	\\
	& \le & |f(\zeta)-f(\zeta_j)|+|f(\zeta_j)-f_j(\zeta_j)|\quad \text{for all $j\ge i$}.
\end{eqnarray*}
Since $f=\lim_{j\to\infty} f_j$ and $\zeta=\lim_{j\to\infty} \zeta_j$, the continuity of $f$ ensures that the right-hand term of this inequality goes to zero as $j$ goes to infinity, and hence $f(\zeta)=z_0$. This proves the former inclusion in \eqref{eq:surjective}; the latter one is obvious. Together with \eqref{eq:limnu}, this shows that $f\colon\b_n\to\c^q$ is surjective, as required.

Let us check condition {\rm (i)}. Pick an arbitrary point $a_k\in A=\{a_0,a_1,a_2,\ldots\}$ and let $V$ be a connected component of $f^{-1}(a_k)\subset\b_n$. Since $f\colon\b_n\to\c^q$ is surjective we have that $f^{-1}(a_k)\neq\varnothing$, hence $V\neq\varnothing$. Take an integer $j_0\ge k\in\z_+$ so large that $V\cap r_{j_0}\b_n\neq\varnothing$; such a number exists since $\lim_{j\to\infty} r_j=1$ (see \eqref{eq:limrj}). From properties {\rm (3$_j$)}, $j\ge j_0+1\in\n$, and taking into account that the sequence $\{r_j\}_{j\in\z_+}$ is increasing, we infer that 
\begin{equation}\label{eq:fi-1ak}
	f_i^{-1}(a_k)\cap r_{j_0}\overline\b_n= f_{j_0}^{-1}(a_k)\cap r_{j_0}\overline\b_n
	\quad \text{for all $i\ge j_0\ge k$}.
\end{equation}
Since $f=\lim_{j\to\infty} f_j$ (see \eqref{eq:f=limfj}), this implies that
\begin{equation}\label{eq:f-1ak}
	f^{-1}(a_k)\cap r_{j_0}\b_n= f_{j_0}^{-1}(a_k)\cap r_{j_0}\b_n.
\end{equation}
Indeed, the inclusion $f^{-1}(a_k)\cap r_{j_0}\b_n\supset f_{j_0}^{-1}(a_k)\cap r_{j_0}\b_n$ is obvious. For the other inclusion reason by contradiction and assume that there exists $z_0\in r_{j_0}\b_n \setminus f_{j_0}^{-1}(a_k)$ with $f(z_0)=a_k$. Since the set  $r_{j_0}\b_n \setminus f_{j_0}^{-1}(a_k)$ is open, $a_k=f(z_0)\in f(r_{j_0}\b_n \setminus f_{j_0}^{-1}(a_k))$, and $f\colon\b_n\to\c^q$ is a holomorphic submersion, hence an open map, it turns out that there is a number $s>0$ such that $a_k+s\b_q\subset f(r_{j_0}\b_n \setminus f_{j_0}^{-1}(a_k))\subset\c^q$. Thus, since $f=\lim_{j\to\infty}f_j$ and each $f_j\colon\c^n\to\c^q$ is a holomorphic submersion, there is an integer $i\ge j_0$ so large that $a_k+ \frac{s}2\b_q\subset f_i( r_{j_0}\b_n \setminus f_{j_0}^{-1}(a_k))$, and hence $a_k\in f_i( r_{j_0}\b_n \setminus f_{j_0}^{-1}(a_k))$. It follows that $f_i^{-1}(a_k)\cap r_{j_0}\b_n\not\subset f_{j_0}^{-1}(a_k)$, which contradicts \eqref{eq:fi-1ak}, thereby showing that $f^{-1}(a_k)\cap r_{j_0}\b_n\subset f_{j_0}^{-1}(a_k)\cap r_{j_0}\b_n$. This proves \eqref{eq:f-1ak}.

Now, since $V$ is a component of $f^{-1}(a_k)$, \eqref{eq:f-1ak} ensures that $V\cap r_{j_0}\b_n$ is a (nonempty) union of connected components of $f_{j_0}^{-1}(a_k)\cap r_{j_0}\b_n$. Therefore, by the connectedness of $V$ and the identity principle, we have that 
\begin{equation}\label{eq:component}
	\text{$V$ is a component of $f_{j_0}^{-1}(a_k)\cap\b_n=(f_{j_0}|_{\b_n})^{-1}(a_k)$},
\end{equation}
which prevents $V$ to be complete. Indeed, choose a path $\gamma\colon[0,1]\to f_{j_0}^{-1}(a_k)$ with $\gamma(0)\in V\subset f_{j_0}^{-1}(a_k)$, $|\gamma(1)|>1$, and $\length(\gamma)<+\infty$; such a path exists by \eqref{eq:component} and the fact that the smooth complex submanifold $f_{j_0}^{-1}(a_k)$ is closed in $\c^n$. Thus, there is a (unique) $t_0\in(0,1)$ such that $|\gamma(t_0)|=1$ and $|\gamma(t)|<1$ for all $t\in [0,t_0)$. By connectedness and taking into account \eqref{eq:component} once again, it turns out that $\gamma|_{[0,t_0)}$ assumes values in $V$ and $\gamma|_{[0,t_0)}\colon[0,t_0)\to V$ is a divergent path with finite length; note that $\length(\gamma|_{[0,t_0)})<\length(\gamma)<+\infty$. This implies that  $V$ is not complete, as claimed, and proves {\rm (i)}.

Finally, let us check condition {\rm (ii)}. Recall that $C=\{c_0,c_1,c_2,\ldots\}$ and pick an arbitrary point $c_k\in C$. 
We claim first that
  \begin{equation}\label{eq:f-1ckL}
	f^{-1}(c_k)\cap L_i=\varnothing\quad\text{for all $i\ge k$}.
\end{equation}
Indeed, reason by contradiction and assume that $c_k\in f(L_i)$ for some $i\ge k$. Since $L_0=\varnothing$, we have that $i\ge 1$. Arguing as above, since $O_i$ is an open neighborhood of $L_i$ in $\b_n$ and $f\colon\b_n\to\c^q$ and each $f_j\colon\c^n\to\c^q$ are holomorphic submersions, hence open maps, there is an integer $l\ge i\ge 1$ so large that $c_k\in f_l(O_i)$. Taking into account that $i\ge k$, this implies that $c_k\in f_l(O_i)\cap \{c_0,\ldots,c_i\}$ and so this intersection is not empty. Since $l\ge i$, this contradicts {\rm (7$_l$)} and thereby proves \eqref{eq:f-1ckL}. 

Next, we claim that \eqref{eq:f-1ckL} and properties {\rm (6$_j$)}, $j\ge 1$, guarantee that every proper path $\gamma\colon [0,1[\to\b_n$ assuming values in $f^{-1}(c_k)$ has infinite Euclidean length. Indeed, let $\gamma$ be such a path and take an integer $l\ge k$ so large that $|\gamma(0)|<r_l$; since $\gamma(0)\in\b_n$, such an $l$ exists by \eqref{eq:limrj}. By properness of $\gamma$ there is a sequence $0<s_1<t_1<s_2<t_2<\cdots$ such that $\lim_{i\to\infty} s_i=\lim_{i\to\infty} t_i=1$ and $|\gamma(s_i)|< r_{l+i}$ and $|\gamma(t_i)|> R_{l+i}$ for all $i\in \n$. 
It follows that
\[
	\length(\gamma)\ge \sum_{i\ge 1} \length(\gamma|_{[s_i,t_i]}) 
	\stackrel{\text{\eqref{eq:f-1ckL},{\rm (6$_{l+i}$)}}}{\ge} \sum_{i\ge 1}1=+\infty,
\]
where in the last inequality we have used that $l+i\ge 1$. Therefore, the smooth closed complex submanifold $f^{-1}(c_k)$ of $\b_n$ is complete.  This proves condition {\rm (ii)} and concludes the proof of Theorem \ref{th:intro-main} under the assumption that Claim \ref{cl:induction} holds.


\begin{proof}[Proof of Claim \ref{cl:induction}]
We proceed by induction. The basis is provided by the tuple $S_0$ formed by the holomorphic submersion $f_0\colon \c^n\to\c^q$, the number $\epsilon_0>0$, the sets $L_0=\varnothing$ and $O_0=\varnothing$, and the ball $\Omega_0$ satisfying \eqref{eq:f0} which were fixed prior to stating the claim. Note that condition {\rm (2$_0$)} is ensured by \eqref{eq:f0} and that {\rm (7$_0$)} is trivially satisfied.
 Conditions {\rm (1$_0$)}, {\rm (3$_0$)}, {\rm (4$_0$)}, and {\rm (5$_0$)} are vacuous whereas {\rm (6$_0$)} does not hold true; nevertheless, this is not any shortcoming since we shall not use them in the construction of the tuple $S_1$. 
Thus, for the inductive step we fix an integer $j\in\n$, assume that for each $i\in\{0,\ldots, j-1\}$ we already have a tuple $S_i=\{f_i,\epsilon_i,L_i,O_i,\Omega_i\}$ for which {\rm (2$_i$)} and {\rm (7$_i$)} are satisfied, and let us construct a tuple $S_j=\{f_j,\epsilon_j,L_j,O_j,\Omega_j\}$ satisfying all conditions {\rm (1$_j$)}--{\rm (7$_j$)}. 

First of all fix $\epsilon_j>0$ satisfying conditions {\rm (4$_j$)} and {\rm (5$_j$)}; such a number exists by the Cauchy estimates in view of the submersivity of $f_{j-1}\colon \c^n\to\c^q$ and the compactness of $\rho_{j-1}\overline\b_n\Subset r_j\overline\b_n$. 

Denote
\begin{equation}\label{eq:V}
	V=f_{j-1}^{-1}(\{a_0,\ldots,a_j\})\subset \c^n
\end{equation}
and note that $V$ is a smooth closed complex submanifold in $\c^n$ of pure codimension $q$ which is not empty in view of \eqref{eq:lambda>} and condition {\rm (2$_{j-1}$)}. To be more precise, these conditions and the fact that $\overline\Omega_i\subset \rho_i\b_n$ for all $i\in\{0,\ldots,j-1\}$ ensure that $f_{j-1}^{-1}(a_i)\cap \rho_i\b_n\neq\varnothing$ for every such $i$; on the other hand, the set $f_{j-1}^{-1}(a_j)$ could be empty. In particular, we have that 
\[
	V\cap \rho_{j-1}\b_n\neq\varnothing. 
\]
By continuity of $f_{j-1}$ and compactness of $V\cap R_j\overline\b_n\supset V\cap \rho_{j-1}\b_n\neq\varnothing$ and of $\{c_0,\ldots,c_j\}$, there is a positive number $\eta>0$ so small that
\begin{equation}\label{eq:eta}
	f_{j-1}(z)\in\{c_0,\ldots,c_j\}\quad
	\text{for no $z\in R_j\overline\b_n$ with $\dist(z,V)\le\eta$};
\end{equation}
recall the definition of $V$ in \eqref{eq:V} and that the sets $A=\{a_0,a_1,a_2,\ldots\}$ and $C=\{c_0,c_1,c_2,\ldots\}$ are disjoint.

Let $L_j$ be a tangent labyrinth in the spherical shell $R_j\b_n\setminus r_j\overline \b_n$ (see Definition \ref{def:labyrinth}) satisfying condition {\rm (6$_j$)} and such that
\begin{equation}\label{eq:diamLj}
	\diam(T)<\eta\quad \text{for every connected component $T$ of $L_j$}.
\end{equation}
Existence of such a tangent labyrinth is ensured by \cite[Lemma 2.3]{Alarcon2018Foliations}. Denote by $\Lambda_V$ the union of all connected components of $L_j$ having nonempty intersection with $V$ and set $\Lambda_0=L_j\setminus \Lambda_V$. (It could be $\Lambda_V=\varnothing$ or $\Lambda_0=\varnothing$, this is not any restriction.) Obviously, both $\Lambda_V$ and $\Lambda_0$ are compact,
\begin{equation}\label{eq:Lj}
	L_j=\Lambda_V\cup\Lambda_0,
	\quad\text{and}\quad
	\Lambda_V\cap\Lambda_0=\varnothing.
\end{equation}
Since, by definition, $\Lambda_0\cap V=\varnothing$, we have that $f_{j-1}(\Lambda_0)\cap \{a_0,\ldots,a_j\}=\varnothing$
in view of \eqref{eq:V}.
On the other hand, \eqref{eq:eta}, \eqref{eq:diamLj}, and the fact that every component of $\Lambda_V$ intersects $V$ imply that
\begin{equation}\label{eq:LambdaV}
	f_{j-1}(\Lambda_V)\cap \{c_0,\ldots,c_j\}=\varnothing.
\end{equation}

Since $\rho_j>R_j$ and  $V$ is a closed complex submanifold of $\c^n$, we have that the set $\rho_j\b_n\setminus (R_j\overline\b_n\cup V)$ is open and nonempty.
Choose an open Euclidean ball $\Omega_j$ in $\c^n$ with 
\begin{equation}\label{eq:Omegaj}
	\overline\Omega_j\subset \rho_j\b_n\setminus (R_j\overline\b_n\cup V).
\end{equation}	
 We claim that
\begin{equation}\label{eq:Kallin}
	\text{$r_j\overline\b_n\cup L_j\cup \overline\Omega_j$ is a polynomially convex compact set in $\c^n$.}
\end{equation}
Indeed, recall first that, since $L_j$ is a tangent labyrinth in $R_j\b_n\setminus r_j\overline \b_n$, the compact set $r_j\overline\b_n\cup L_j\subset R_j\b_n$ is polynomially convex as it was pointed out in \cite{AlarconForstneric2017PAMS}; see also \cite[Remark 2.4]{Alarcon2018Foliations}. On the other hand, a standard application of Kallin's lemma (see \cite{Kallin1965} or \cite[p.\ 62]{Stout2007PM}) and the Oka-Weil theorem (see \cite[Theorem 1.5.1]{Stout2007PM}) shows that if $X$ and $Y$ are a pair of disjoint compact convex sets in $\c^n$ and $K\subset X$ is a compact polynomially convex set, then the union $K\cup Y$ is polynomially convex. Taking into account \eqref{eq:Omegaj}, this well-known fact applies with $X=R_j\overline\b_n$, $Y=\overline\Omega_j$, and $K=r_j\overline\b_n\cup L_j$, thereby proving \eqref{eq:Kallin}.

We now claim the following.
\begin{claim}\label{cl:2}
There are an open neighborhood $W$ of $r_j\overline\b_n\cup V\cup L_j\cup \overline\Omega_j$ in $\c^n$ and a holomorphic submersion $\psi\colon W\to\c^q$ satisfying the following properties.
\begin{itemize}
\item[\rm (A)] $\psi(z)=f_{j-1}(z)$ for all $z\in r_j\overline\b_n\cup V\cup \Lambda_V$. In particular, the same equality holds at every point $z$ in the connected component of $W$ containing $r_j\overline\b_n\cup V\cup \Lambda_V$.

\smallskip
\item[\rm (B)] $\psi(\Lambda_0)\cap \{c_0,\ldots,c_j\}=\varnothing$.

\smallskip
\item[\rm (C)] $\lambda_j\overline\b_q\subset \psi(\Omega_j)$.

\smallskip
\item[\rm (D)] There is a $q$-coframe $\theta_j$ on $\c^n$ with  $\theta_j|_W=d\psi$ (see Definition \ref{def:qcoframe}).
\end{itemize}
\end{claim}
We defer the proof of Claim \ref{cl:2} to later on. Granted this claim, the proof of Claim \ref{cl:induction} is completed as follows.

Let $W$ and $\psi\colon W\to\c^q$ be an open neighborhood of $r_j\overline\b_n\cup V\cup L_j\cup\overline\Omega_j$ in $\c^n$ and a holomorphic submersion, respectively, satisfying the conclusion of Claim \ref{cl:2}. 
Fix a number $\epsilon'>0$ to be specified later. 
In view of \eqref{eq:Kallin} and condition {\rm (D)} in Claim \ref{cl:2}, we are allowed to apply the Forstneri\v c-Oka-Weil-Cartan theorem for holomorphic submersions \cite[Theorem 2.5]{Forstneric2003AM} with the Stein manifold $\c^n$, the integer $q\in\{1,\ldots,n-1\}$, the $q$-coframe $\theta_j$ on $\c^n$, the closed complex subvariety $V\subset W$ of $\c^n$, the polynomially convex compact set $r_j\overline\b_n\cup L_j\cup\overline\Omega_j\subset W$, the holomorphic submersion $\psi\colon W\to\c^q$, the number $\epsilon'>0$, and an integer $l\in\n$. This furnishes us with a holomorphic submersion $f_j\colon \c^n\to\c^q$ satisfying the following properties.
\begin{itemize}
\item[\rm (a)] $|f_j(z)-\psi(z)|<\epsilon'$ for all $z\in r_j\overline\b_n\cup L_j\cup\overline\Omega_j$.

\smallskip
\item[\rm (b)] $f_j-\psi$ vanishes to order $l\ge 1$ on $V=f_{j-1}^{-1}(\{a_0,\ldots,a_j\})$ (see \eqref{eq:V}).
\end{itemize}

Assuming that $\epsilon'>0$ is chosen small enough, the following assertions are satisfied.
\begin{itemize}
\item Conditions {\rm (A)} and {\rm (a)} imply {\rm (1$_j$)}.

\smallskip
\item Taking into account that $\overline\Omega_i\subset r_j\b_n$ for all $i\in\{0,\ldots,j-1\}$, properties {\rm (2$_{j-1}$)}, {\rm (A)}, {\rm (C)}, and {\rm (a)} ensure {\rm (2$_j$)}.

\smallskip
\item Conditions \eqref{eq:V}, {\rm (A)}, {\rm (a)}, and {\rm (b)} guarantee {\rm (3$_j$)}.

\smallskip
\item Taking into account that the sets $\{c_0,\ldots,c_i\}$ and $\overline O_i\subset r_j\b_n$ are both compact for all $i\in\{0,\ldots,j-1\}$, properties {\rm (7$_{j-1}$)}, {\rm (A)}, and {\rm (a)} ensure that $f_j(\overline O_i)\cap\{c_0,\ldots,c_i\}=\varnothing$ for every such $i$; that is, the statement in {\rm (7$_j$)} is satisfied for all $i<j$.
\end{itemize}
Thus, since conditions {\rm (4$_j$)}, {\rm (5$_j$)}, and {\rm (6$_j$)} concerning the number $\epsilon_j>0$ and the tangent labyrinth $L_j\subset R_j\b_n\setminus r_j\overline\b_n$ have been already ensured above, in order to conclude the construction of a tuple $S_j=\{f_j,\epsilon_j,L_j,O_j,\Omega_j\}$ satisfying all properties {\rm (1$_j$)}--{\rm (7$_j$)}, granted the existence of $W$ and $\psi$, it suffices to find an open neighborhood $O_j$ of $L_j$ in $\c^n$, with $\overline O_j\subset R_j\b_n\setminus r_j\overline\b_n$, for which $f_j(\overline O_j)\cap \{c_0,\ldots,c_j\}=\varnothing$ holds whenever $\epsilon'>0$ is small enough; indeed, that would complete the proof of {\rm (7$_j$)}. Nevertheless, conditions \eqref{eq:Lj}, \eqref{eq:LambdaV}, {\rm (A)}, {\rm (B)}, and {\rm (a)} ensure that $f_j(L_j)\cap \{c_0,\ldots,c_j\}=\varnothing$ provided that $\epsilon'>0$ is sufficiently small, hence such a neighborhood $O_j$ trivially exists; recall that $\Lambda_V$, $\Lambda_0$, and $\{c_0,\ldots,c_j\}$ are all compact.

This closes the induction and concludes the proof of Claim \ref{cl:induction} under the assumption that Claim \ref{cl:2} holds true.


\begin{proof}[Proof of Claim \ref{cl:2}]
We first define the set $W$.
Recall that $L_j=\Lambda_V\cup\Lambda_0\subset R_j\b_n\setminus r_j\overline\b_n$ and $\overline \Omega_j\subset \rho_j\b_n\setminus R_j\overline\b_n$. In particular, we have that $L_j\cap\overline \Omega_j=\varnothing$. Also recall that both $\overline\Omega_j$ and $\Lambda_0$ are disjoint from $r_j\overline\b_n\cup V$ (see \eqref{eq:Omegaj} and the definition of $\Lambda_0$ just above \eqref{eq:Lj}, and take into account that $r_j<R_j<\rho_j$).
We shall choose the open set $W$ of the form
\begin{equation}\label{eq:W}
	W=(\c^n\setminus \Delta_2)\cup \mathring\Delta_1,
\end{equation}
where $\Delta_1$ and $\Delta_2$ are compact tubular neighborhoods of $\Lambda_0\cup\overline\Omega_j$ in $\c^n$ satisfying the following properties.
\begin{itemize}
\item $\Delta_1\Subset \Delta_2\Subset \c^n$.

\smallskip
\item Each component $U_2$ of $\Delta_2$ contains a unique component $U_1$ of $\Delta_1$. Moreover, there is a homeomorphism from $U_2$ to $2\overline \b_n$ mapping $U_1$ homeomorphically into $\overline\b_n$.

\smallskip
\item Each component of $\Delta_1$ intersects a unique component of $\Lambda_0\cup\overline\Omega_j$, which is contained in its interior.

\smallskip
\item $(r_j\overline\b_n\cup V\cup \Lambda_V)\cap \Delta_2=\varnothing$.
\end{itemize}
Such neighborhoods $\Delta_1$ and $\Delta_2$ clearly exist since $\Lambda_0\cup\overline\Omega_j$ has finitely many connected components all which are compact and convex, the set $r_j\overline\b_n$ is compact, the set $V$ is a closed complex submanifold of $\c^n$, and $(\Lambda_0\cup\overline\Omega_j)\cap(r_j\overline\b_n\cup V)=\varnothing$; we refer to \cite[proof of Claim 3.4]{Alarcon2018Foliations} for the details in the slightly simpler case when $\overline\Omega_j=\varnothing$. In our case, we just need to take $\Delta_1$ and $\Delta_2$ with one more component containing the ball $\overline\Omega_j$. 

Fix $W$ as in \eqref{eq:W} and satisfying the aforementioned conditions. It turns out that $\c^n\setminus \Delta_2$ is a connected open neighborhood of $r_j\overline\b_n\cup V\cup \Lambda_V$, $\mathring \Delta_1$ is an open neighborhood of $\Lambda_0\cup\overline\Omega_j$, and $(\c^n\setminus\mathring \Delta_2)\cap \Delta_1=\varnothing$. Thus, since $L_j=\Lambda_V\cup\Lambda_0$, we have that $W$ is an open neighborhood of $r_j\overline\b_n\cup V\cup L_j\cup\overline\Omega_j$, as required. 

We now turn to define the holomorphic submersion $\psi\colon W\to\c^q$. Since the map $f_{j-1}\colon\c^n\to\c^q$ is continuous and the sets $\Lambda_0\subset\c^n$ and $\{c_0,\ldots,c_j\}$ are compact, there is $\xi_0\in\c^q$ such that 
\begin{equation}\label{eq:xi0}
	(\xi_0+f_{j-1}(\Lambda_0))\cap\{c_0,\ldots,c_j\}=\varnothing.
\end{equation}
Indeed, we may for instance choose any $\xi_0\in\c^q$ with $|\xi_0|>2\max\{|z|\colon z\in f_{j-1}(\Lambda_0)\cup\{c_0,\ldots,c_j\}\}$.
On the other hand, since $f_{j-1}$ is a holomorphic submersion, hence an open map, there are $\xi_1\in\c^q$ and $\tau_1>1$ such that
\begin{equation}\label{eq:xi1}
	\lambda_j\overline\b_q\subset \xi_1+\tau_1 f_{j-1}(\Omega_j).
\end{equation}
Indeed, we may for instance take first $z\in\c^q$ with $-z\in f_{j-1}(\Omega_j)$. Thus, $z+f_{j-1}(\Omega_j)$ contains the origin, hence also the closed ball $r\overline\b_q$ for some $r>0$. Then, it suffices to choose $\tau_1>\max\{1,\lambda_j/r\}$ and $\xi_1=\tau_1z$.

Denote by $\Delta_1^\Omega$ the connected component of $\Delta_1$ containing $\overline\Omega_j$. We have that $\mathring \Delta_1\setminus \Delta_1^\Omega$ is an open neighborhood of $\Lambda_0$. We define $\psi\colon W\to\c^q$ by
\[
	W\ni z\longmapsto \psi(z)=\left\{
	\begin{array}{ll}
	f_{j-1}(z) & \text{for all $z\in \c^n\setminus \Delta_2$}\smallskip
	\\
	\xi_0+ f_{j-1}(z) & \text{for all $z\in\mathring \Delta_1\setminus \Delta_1^\Omega$}\smallskip
	\\
	\xi_1+\tau_1f_{j-1}(z) & \text{for all $z\in \mathring\Delta_1^\Omega$}.
	\end{array}\right.
\]
Since $(\c^n\setminus \mathring \Delta_2)\cap \Delta_1=\varnothing$, we have that $\psi$ is well defined and continuous.
Further, since $f_{j-1}$ is submersive everywhere in $\c^n$ and $\xi_0$, $\xi_1$, and $\tau_1\neq 0$ are constants, we infer that $\psi$ is a holomorphic submersion. Conditions {\rm (A)}, {\rm (B)}, and {\rm (C)} in the statement of the claim clearly follow from the definition of $\psi$ and the choice of $\xi_0$ in \eqref{eq:xi0} and of $\xi_1$ and $\tau_1$ in \eqref{eq:xi1}. 

It remains to construct a $q$-coframe $\theta_j$ on $\c^n$ with $\theta_j|_W=d\psi$. For that we denote by $\Delta_2^\Omega$ the connected component of $\Delta_2$ containing $\Delta_1^\Omega$ and take any continuous map $\mu\colon \Delta_2^\Omega\setminus \mathring \Delta_1^\Omega\to [1,\tau_1]$ such that 
\begin{equation}\label{eq:mu}
 \left\{\begin{array}{ll}
 \mu(z)=1 & \text{for all $z\in  b \Delta_2^\Omega$}\smallskip
 \\
 \mu(z)=\tau_1 & \text{for all $z\in  b \Delta_1^\Omega$};
 \end{array}\right.
\end{equation}
recall that $\tau_1>1$.
Such a function clearly exists since $\Delta_2^\Omega\setminus\mathring \Delta_1^\Omega$ is a closed spherical shell in $\c^n$; one can also invoke Urysohn's lemma. We define
\[
	\theta_j=(\theta_{j,1},\ldots,\theta_{j,q})=
\left\{\begin{array}{ll}
	df_{j-1} & \text{on $\c^n\setminus\Delta_2^\Omega$}\smallskip
	\\
	\mu\, df_{j-1} & \text{on $\Delta_2^\Omega\setminus\mathring \Delta_1^\Omega$}\smallskip
	\\
	\tau_1\, df_{j-1} & \text{on $\mathring\Delta_1^\Omega$}.	
\end{array}\right.
\]
By \eqref{eq:mu}, $\theta_j$ is a continuous $1$-form on $\c^n$. Moreover, since $f_{j-1}$ is submersive everywhere on $\c^n$, $\tau_1\neq 0$, and $\mu$ vanishes nowhere on $\Delta_2^\Omega\setminus \mathring \Delta_1^\Omega$, we have that $\theta_{j,1},\ldots,\theta_{j,q}$ are linearly independent at every point in $\c^n$, that is, $\theta_j$ is a $q$-coframe on $\c^n$. Finally, it trivially follows from the definitions of $W$, $\psi$, and $\theta_j$ that $\theta_j|_W=d\psi$; recall that $\xi_0$, $\xi_1$, and $\tau_1$ are constants. This guarantees condition {\rm (D)} in the statement of Claim \ref{cl:2}, thereby concluding the proof.
\end{proof}
The proof of Claim \ref{cl:induction} is complete.
\end{proof}
Theorem \ref{th:intro-main} is proved.


\section{Proof of Theorem \ref{th:intro}}\label{sec:proof}

\noindent The proof follows closely that of Theorem \ref{th:intro-main}; we explain the modifications that need to be made. Let $n$, $q$, $\wt A$, $\wt C$, $A$, and $C$ be as in the statement of Theorem \ref{th:intro}. By the assumptions, both $\wt C$ and $A$ are countable sets; we assume that they both are infinite (otherwise the proof is even simpler) and write
\[
	A=\{a_0,a_1,a_2,\ldots\} \quad\text{and}\quad
	\wt C=\{c_0,c_1,c_2,\ldots\}.
\]
Choose sequences of positive numbers $\lambda_j$, $r_j$, $R_j$, and $\rho_j$, $j\in\z_+$, an open Euclidean ball $\Omega_0$ in $\c^n$, and a holomorphic submersion $f_0\colon\c^n\to\c^q$ as those in the proof of Theorem \ref{th:intro-main}. In particular, we assume that conditions \eqref{eq:lambda>}--\eqref{eq:f0} are satisfied. We also fix $\epsilon_0>0$ and call $L_0=O_0=\varnothing$.
\begin{claim}\label{lem:1}
There is a sequence $S_j=\{f_j,\epsilon_j,L_j,O_j,\Omega_j\}$ $(j\in\n)$, where $f_j$, $\epsilon_j$, $L_j$, $O_j$, and $\Omega_j$ are as in Claim \ref{cl:induction}, such that conditions {\rm (1$_j$)}--{\rm (6$_j$)} in that claim and the following properties are satisfied for all $j\ge 1$.
\begin{itemize}
\item[\rm (7$'_j$)] $f_j(\overline O_i)\cap \wt C=\varnothing$ for all $i\in\{0,\ldots,j\}$. 

\smallskip
\item[\rm (8$'_j$)] If $z\in L_j$, then either $|f_j(z)|>\lambda_j$ or $\displaystyle\dist(f_j(z),\{a_0,\ldots,a_j\})< \frac1{\lambda_j}$.
\end{itemize}
\end{claim}
\begin{proof}
We explain the induction. Note that {\rm (2$_0$)} and {\rm (7$'_0$)} are satisfied by the fixed tuple $S_0=\{f_0,\epsilon_0,L_0,O_0,\Omega_0\}$ (condition {\rm (8$'_0$)} holds too, but this is not relevant for the argument). Assume that for some integer $j\in\n$ we have for each $i\in\{0,\ldots,j-1\}$ a tuple $S_i$ satisfying {\rm (2$_i$)} and {\rm (7$'_i$)}, and let us provide $S_j$ enjoying conditions {\rm (1$_j$)}--{\rm (6$_j$)}, {\rm (7$'_j$)}, and {\rm (8$'_j$)}. 

Fix $\epsilon_j>0$ satisfying {\rm (4$_j$)} and {\rm (5$_j$)}, define $V$ as in \eqref{eq:V}, and choose $\eta>0$ so small that
\begin{equation}\label{eq:eta22}
	\dist(f_{j-1}(z),\{a_0,\ldots,a_j\})< \frac1{\lambda_j}\quad
	\forall z\in R_j\overline\b_n \text{ with $\dist(z,V)\le\eta$}.
\end{equation}
Since $\{a_0,\ldots,a_j\}\cap \wt C=\varnothing$ and $\wt C\subset \c^q$ is closed, we may in addition assume,  by continuity of $f_{j-1}$ and compactness of $V\cap R_j\overline\b_n$, that $\eta>0$ is so small that
\begin{equation}\label{eq:eta'22}
	f_{j-1}(z)\in \wt C\quad 
	\text{for no $z\in R_j\overline\b_n$ with $\dist(z,V)\le\eta$}.
\end{equation}  
Let $L_j$, $\Lambda_V$, and $\Lambda_0$ be defined as in the proof of Claim \ref{cl:induction}; in particular, $L_j$ is a tangent labyrinth in $R_j\b_n\setminus r_j\overline \b_n$ satisfying conditions {\rm (6$_j$)}, \eqref{eq:diamLj}, and \eqref{eq:Lj}. By \eqref{eq:eta22}, \eqref{eq:eta'22}, and \eqref{eq:diamLj}, we have that
\begin{equation}\label{eq:LambdaV22}
	\dist(f_{j-1}(z),\{a_0,\ldots,a_j\})< \frac1{\lambda_j}\quad 
	\text{for all $z\in \Lambda_V$}
\end{equation}
and
\begin{equation}\label{eq:LambdaV'22}
	f_{j-1}(\Lambda_V)\cap \wt C=\varnothing.
\end{equation}
Also choose an open Euclidean ball $\Omega_j$ in $\c^n$ satisfying \eqref{eq:Omegaj}.
\begin{claim}\label{cl:222}
There are an open neighborhood $W$ of $r_j\overline\b_n\cup V\cup L_j\cup \overline\Omega_j$ in $\c^n$ and a holomorphic submersion $\psi\colon W\to\c^q$ satisfying
\begin{itemize}
\item[\rm (B$'$)] $\psi(\Lambda_0)\cap (\wt C\cup\lambda_j\overline\b_q)=\varnothing$
\end{itemize}
as well as conditions {\rm (A)}, {\rm (C)}, and {\rm (D)} in Claim \ref{cl:2}.
\end{claim}
\begin{proof}
We choose $W$, $\Delta_1$, $\Delta_2$, $\Delta_1^\Omega$, and $\Delta_2^\Omega$ as in the proof of Claim \ref{cl:2}; in particular, $W$ is of the form \eqref{eq:W}.  Take a point $\xi_0\in \c^q\setminus \wt C$ with $|\xi_0|>\lambda_j$ and fix a number $\tau_0$, $0<\tau_0<1$, so small that
\begin{equation}\label{eq:clambda22}
	\xi_0+\tau_0 f_{j-1}(\Lambda_0)\subset \{z\in\c^q\setminus \wt C\colon |z|>\lambda_j\}. 
\end{equation}
Existence of such a $\tau_0$ is ensured by compactness of $\Lambda_0$ and continuity of $f_{j-1}$; recall that $\wt C\subset\c^q$ is closed, and so $\dist(\xi_0,\wt C)>0$. We also choose $\xi_1\in\c^q$ and $\tau_1>1$ satisfying \eqref{eq:xi1} and define $\psi\colon W\to\c^q$ by
\[
	W\ni z\longmapsto \psi(z)=\left\{
	\begin{array}{ll}
	f_{j-1}(z) & \text{for all $z\in \c^n\setminus \Delta_2$}\smallskip
	\\
	\xi_0+\tau_0 f_{j-1}(z) & \text{for all $z\in\mathring \Delta_1\setminus \Delta_1^\Omega$}\smallskip
	\\
	\xi_1+\tau_1 f_{j-1}(z) & \text{for all $z\in\mathring \Delta_1^\Omega$}.
	\end{array}\right.
\]
Thus, $\psi$ is a well-defined holomorphic submersion that, in view of \eqref{eq:clambda22} and \eqref{eq:xi1}, satisfies conditions {\rm (A)}, {\rm (B$'$)}, and {\rm (C)}. 

In order to construct a $q$-coframe $\theta_j$ on $\c^n$ satisfying {\rm (D)} we first take any continuous map $\mu\colon \Delta_2\setminus \mathring \Delta_1\to [\tau_0,\tau_1]$ such that 
\begin{equation}\label{eq:mu22}
 \left\{\begin{array}{ll}
 \mu(z)=1 & \text{for all $z\in  b \Delta_2$}\smallskip
 \\
 \mu(z)=\tau_0 & \text{for all $z\in  b \Delta_1\setminus b \Delta_1^\Omega$}\smallskip
 \\
 \mu(z)=\tau_1 & \text{for all $z\in  b \Delta_1^\Omega$};
 \end{array}\right.
\end{equation}
recall that $0<\tau_0<1<\tau_1$.
Such a function clearly exists in view of the topology of $\Delta_2\setminus\mathring \Delta_1$: a finite union of pairwise disjoint closed shells in $\c^n$. Finally, we define
\[
	\theta_j=(\theta_{j,1},\ldots,\theta_{j,q})=\left\{\begin{array}{ll}
	df_{j-1} & \text{on $\c^n\setminus\Delta_2$}\smallskip
	\\
	\mu\, df_{j-1} & \text{on $\Delta_2\setminus\mathring \Delta_1$}\smallskip
	\\
	\tau_0\, df_{j-1} & \text{on $\mathring\Delta_1\setminus \Delta_1^\Omega$}\smallskip
	\\
	\tau_1\,df_{j-1} & \text{on $\mathring\Delta_1^\Omega$}.
	\end{array}\right.
\]
It turns out that $\theta_j$ is a $q$-coframe on $\c^n$ with $\theta_j|_W=d\psi$. This shows condition {\rm (D)} and concludes the proof of the claim.
\end{proof}

Pick $W$ and $\psi$ as in Claim \ref{cl:222} and fix a number $\epsilon'>0$ to be specified later. By \cite[Theorem 2.5]{Forstneric2003AM}, there is a holomorphic submersion $f_j\colon \c^n\to\c^q$ satisfying\begin{itemize}
\item[\rm (a)] $|f_j(z)-\psi(z)|<\epsilon'$ for all $z\in r_j\overline\b_n\cup L_j\cup\overline\Omega_j$ and

\smallskip
\item[\rm (b)] $f_j-\psi$ vanishes to order $l\ge 1$ everywhere on $V$.
\end{itemize}
Assuming that $\epsilon'>0$ is chosen sufficiently small, conditions {\rm (1$_j$)}--{\rm (6$_j$)} are seen as in the proof of Claim \ref{cl:induction}. Moreover, {\rm (8$'_j$)} follows from \eqref{eq:lambda>}, \eqref{eq:Lj}, \eqref{eq:LambdaV22}, {\rm (A)}, {\rm (B$'$)}, and {\rm (a)} provided that $\epsilon'>0$ is small enough. Finally, by conditions \eqref{eq:LambdaV'22}, {\rm (B$'$)}, and {\rm (a)} and the facts that $f_j$ is continuous, $L_j$ is compact, and $\wt C$ is a closed discrete subset of $\c^q$, we may choose $\epsilon'>0$ so small that, in addition to the above, $f_j(L_j)\cap \wt C=\varnothing$. Therefore, there is an open neighborhood $O_j$ of $L_j$ with $\overline O_j\subset R_j\b_n\setminus r_j\overline\b_n$ such that $f_j(\overline O_j)\cap\wt C=\varnothing$. Together with {\rm (7$'_{j-1}$)}, {\rm (A)}, and {\rm (a)}, this shows that condition {\rm (7$'_j$)} is satisfied whenever that $\epsilon'>0$ is chosen sufficiently small. Summarizing, the tuple $S_j=\{f_j,\epsilon_j,L_j,O_j,\Omega_j\}$ satisfies all properties {\rm (1$_j$)}--{\rm (6$_j$)}, {\rm (7$'_j$)}, and {\rm (8$'_j$)}. This closes the induction process and completes the proof of Claim \ref{lem:1}.
\end{proof}

As in the proof of Theorem \ref{th:intro-main}, there is a limit surjective holomorphic submersion
\[
	f=\lim_{j\to\infty}f_j\colon\b_n\to\c^q
\]
which satisfies \eqref{eq:f-fj-1} and condition {\rm (i)} in the statement of Theorem \ref{th:intro}. 
Let us check condition {\rm (ii)}. Choose a point $c\in C=\c^q\setminus A$. If $c\in \wt C$, then reasoning as in the proof of Theorem \ref{th:intro-main} properties {\rm (7$'_j$)} ensure that $f^{-1}(c)$ is complete. 
Assume now that $c\notin \wt C$, hence $c\in \c^q\setminus(A\cup \wt C)$. By the assumptions, $A$ is a closed subset of $\c^q\setminus \wt C$, hence $\dist(c,A)>0$. Choose $j_0\in\n$ so large that 
\begin{equation}\label{eq:j0}
	|c|+\epsilon_j<\lambda_j
	\quad\text{and}\quad
	 \frac1{\lambda_j}+\epsilon_j<\dist(c,A)\quad \text{for all $j\ge j_0$}.
\end{equation}
Such a number exists by \eqref{eq:limnu} and the fact that $\lim_{j\to\infty}\epsilon_j=0$ (see properties {\rm (4$_j$)}). If $z\in L_j$ for some $j\in\z_+$, then {\rm (8$'_j$)} and \eqref{eq:f-fj-1} imply that either $|f(z)|>\lambda_j-\epsilon_j$ or $\dist(f(z),A)\le\dist(f(z),\{a_0,\ldots,a_j\})<1/\lambda_j+\epsilon_j$. This and \eqref{eq:j0} show that 
$f^{-1}(c)\cap L_j=\varnothing$ for all $j\ge j_0$,
which, by properties {\rm (6$_j$)}, ensures that $f^{-1}(c)$ is complete. This proves condition {\rm (ii)} and concludes the proof of Theorem \ref{th:intro}.


\subsection*{Acknowledgements}
This research was partially supported by the State Research Agency (AEI) 
via the grants no.\ MTM2017-89677-P and PID2020-117868GB-I00, and the ``Maria de Maeztu'' Excellence Unit IMAG, reference CEX2020-001105-M, funded by MCIN/AEI/10.13039/501100011033/; the Junta de Andaluc\'ia grant no. P18-FR-4049; and the Junta de Andaluc\'ia - FEDER grant no. A-FQM-139-UGR18; Spain.



\begin{thebibliography}{10}

\bibitem{Alarcon2018Foliations}
A.~Alarc{\'o}n.
\newblock {Complete complex hypersurfaces in the ball come in foliations}.
\newblock {\em J. Differential Geom.}, in press.

\bibitem{AlarconForstneric2017PAMS}
A.~Alarc{\'o}n and F.~Forstneri{\v{c}}.
\newblock Complete densely embedded complex lines in {$\mathbb{C}^2$}.
\newblock {\em Proc. Amer. Math. Soc.}, 146(3):1059--1067, 2018.

\bibitem{AlarconForstneric2020MZ}
A.~Alarc{\'o}n and F.~Forstneri{\v{c}}.
\newblock A foliation of the ball by complete holomorphic discs.
\newblock {\em Math. Z.}, 296(1-2):169--174, 2020.

\bibitem{AlarconGlobevnik2017C2}
A.~{Alarc\'on} and J.~{Globevnik}.
\newblock {Complete embedded complex curves in the ball of $\mathbb{C}^2$ can
  have any topology}.
\newblock {\em Anal. PDE}, 10(8):1987--1999, 2017.

\bibitem{AlarconGlobevnikLopez2016Crelle}
A.~Alarc\'{o}n, J.~Globevnik, and F.~J. L\'{o}pez.
\newblock A construction of complete complex hypersurfaces in the ball with
  control on the topology.
\newblock {\em J. Reine Angew. Math.}, 751:289--308, 2019.

\bibitem{CharpentierKosinski2020}
S.~Charpentier and {\L}.~Kosi\'{n}ski.
\newblock Construction of labyrinths in pseudoconvex domains.
\newblock {\em Math. Z.}, 296(3-4):1021--1025, 2020.

\bibitem{CharpentierKosinski2021}
S.~Charpentier and {\L}.~Kosi\'{n}ski.
\newblock Wild boundary behaviour of holomorphic functions in domains of
  {$\mathbb{C}^N$}.
\newblock {\em Indiana Univ. Math. J.}, 70(6):2351--2367, 2021.

\bibitem{doCarmo1992}
M.~P. do~Carmo.
\newblock {\em Riemannian geometry}.
\newblock Mathematics: Theory \& Applications. Birkh\"auser Boston, Inc.,
  Boston, MA, 1992.
\newblock Translated from the second Portuguese edition by Francis Flaherty.

\bibitem{Forstneric2003AM}
F.~Forstneri{\v{c}}.
\newblock Noncritical holomorphic functions on {S}tein manifolds.
\newblock {\em Acta Math.}, 191(2):143--189, 2003.

\bibitem{Forstneric2017}
F.~Forstneri{\v{c}}.
\newblock {\em Stein manifolds and holomorphic mappings. The homotopy principle
  in complex analysis (2nd edn)}, volume~56 of {\em Ergebnisse der Mathematik
  und ihrer Grenzgebiete. 3. Folge. A Series of Modern Surveys in Mathematics}.
\newblock Springer, Berlin, 2017.

\bibitem{Globevnik2015AM}
J.~Globevnik.
\newblock A complete complex hypersurface in the ball of {$\mathbb{C}^N$}.
\newblock {\em Ann. of Math. (2)}, 182(3):1067--1091, 2015.

\bibitem{Globevnik2016MA}
J.~Globevnik.
\newblock Holomorphic functions unbounded on curves of finite length.
\newblock {\em Math. Ann.}, 364(3-4):1343--1359, 2016.

\bibitem{Kallin1965}
E.~Kallin.
\newblock Polynomial convexity: {T}he three spheres problem.
\newblock In {\em Proc. {C}onf. {C}omplex {A}nalysis ({M}inneapolis, 1964)},
  pages 301--304. Springer, Berlin, 1965.

\bibitem{Stout2007PM}
E.~L. Stout.
\newblock {\em Polynomial convexity}, volume 261 of {\em Progress in
  Mathematics}.
\newblock Birkh\"auser Boston, Inc., Boston, MA, 2007.

\bibitem{Yang1977}
P.~Yang.
\newblock Curvature of complex submanifolds of {$C^{n}$}.
\newblock In {\em Several complex variables ({P}roc. {S}ympos. {P}ure {M}ath.,
  {V}ol. {XXX}, {P}art 2, {W}illiams {C}oll., {W}illiamstown, {M}ass., 1975)},
  pages 135--137. Amer. Math. Soc., Providence, R.I., 1977.

\bibitem{Yang1977JDG}
P.~Yang.
\newblock Curvatures of complex submanifolds of {${\bf C}^{n}$}.
\newblock {\em J. Differential Geom.}, 12(4):499--511 (1978), 1977.

\end{thebibliography}


\medskip
\noindent Antonio Alarc\'{o}n

\noindent Departamento de Geometr\'{\i}a y Topolog\'{\i}a e Instituto de Matem\'aticas (IMAG), Universidad de Granada, Campus de Fuentenueva s/n, E--18071 Granada, Spain.

\noindent  e-mail: {\tt alarcon@ugr.es}

\end{document}